\documentclass[11pt]{amsart}

\usepackage[T1]{fontenc}
\usepackage{amsmath,amssymb,mathtools}
\usepackage{microtype}
\usepackage[hidelinks]{hyperref}

\theoremstyle{plain}
\newtheorem{theorem}{Theorem}
\newtheorem{proposition}[theorem]{Proposition}
\newtheorem{lemma}[theorem]{Lemma}
\newtheorem{cor}[theorem]{Corollary}

\theoremstyle{remark}
\newtheorem{remark}[theorem]{Remark}

\newcommand{\cF}{{\mathcal F}}
\newcommand{\cC}{{\mathcal C}}
\newcommand{\cG}{{\mathcal G}}
\newcommand{\cH}{{\mathcal H}}
\newcommand{\cA}{{\mathcal A}}
\newcommand{\bC}{{\mathbf C}}

\newcommand{\LA}{\mathop{}\!\mathrm{La^*}}

\newcommand{\LAw}{\mathop{}\!\mathrm{La}}
\newcommand{\F}{\mathbb F}

\newcommand{\ex}{\mathop{}\!\mathrm{ex}}

\title{On the maximum size of $B_3$-free and $D_s$-free families}
\author{Bal\'azs Patk\'os }
\address{HUN-REN Alfr\'ed R\'enyi Institute of Mathematics and Department of Computer Science and Information Theory, Budapest University of
Technology and Economics} 
\email{patkos@renyi.hu}
\date{}

\begin{document}

\maketitle

\begin{abstract}
    A family $\cG$ of sets is a weak copy of the poset $(P,\leqslant)$ if there exists a bijection $\iota:P\rightarrow \cG$ with $\iota(p)\subset \iota(q)$ whenever $p\leqslant q$. $\cG$ is a strong copy if $\iota(p)\subset \iota(q)$ if and only if $p\leqslant q$ holds. A family is weak (strong) $P$-free if it does not contain any weak (strong) copies of $P$. For a poset $P$, let $e(P)$ ($e^*(P)$) denote the largest positive integer $k$ such that the union of any $k$ consecutive layers of $2^{[n]}$ does not contain a weak (strong) copy of $P$. Ellis, Ivan, and Leader were the first to show the existence of posets $P$ for which there exists a positive real $\varepsilon_P$ such that $\LAw(n,P)\ge (e(P)+\varepsilon_P)\binom{n}{\lfloor n/2\rfloor}$ and $\LA(n,P)\ge (e^*(P)+\varepsilon_P)\binom{n}{\lfloor n/2\rfloor}$ hold, where $\LAw(n,P)$ ($\LA(n,P)$) denotes the maximum size of a weak (strong) $P$-free family $\cF\subseteq 2^{[n]}$. More precisely, they showed that $P=B_d$ are such posets for all $d\ge 4$, where $B_d$ is the Boolean lattice $2^{[d]}$ ordered by inclusion. Very recently, Tompkins showed that the diamond $B_2$ is also such a poset. In this short note, we apply his method to settle the case of the last Boolean poset $B_3$. We show that there exists a positive $\varepsilon$ such that $$\LA(n,B_3)\ge \LAw(n,B_3)\ge \LAw(n,D_6)\ge (3+\varepsilon)\binom{n}{\lfloor n/2\rfloor},$$
    where $D_s$ is the poset on $s+2$ elements $a<b_1,\dots,b_s<c$.

    For any integer $s$, let $m_s=\min\{m:2^m-2\ge s\}$,  $m^*_s=\min\{m:\binom{m}{\lfloor m/2\rfloor}\ge s\}$, and observe $e(D_s)=m_s,e^*(D_s)=m^*_s$. Consider the intervals $I_m=[2^{m-1}-1,2^m-2]$, $I^*_m=[\binom{m-1}{\lfloor \frac{m-1}{2}\rfloor}+1,\binom{m}{\lfloor \frac{m}{2}\rfloor}]$. By results of Griggs, Li, and Lu and of Patk\'os, it is known that for values $s$ in the major initial parts of $I_m$ and $I_m^*$, one has $\LAw(n,D_s)=(m+o(1))\binom{n}{\lfloor \frac{n}{2}\rfloor}$ and $\LA(n,D_s)=(m+o(1))\binom{n}{\lfloor \frac{n}{2}\rfloor}$. The construction of Ellis, Ivan, and Leader shows that the above equalities do not hold for the largest elements of the intervals. As $\LAw(n,D_s)\le \LAw(n,D_{s+1}),\LA(n,D_s)\le \LA(n,D_{s+1})$, there exist   $s_m\in I_m, s^*_m\in I^*_m$ such that for $s\in I_m$ we have $\LAw(n,D_s)=(m+o(1))\binom{n}{\lfloor \frac{n}{2}\rfloor}$ if and only if $s<s_m$ and for $s\in I^*_m$ we have $\LA(n,D_s)=(m+o(1))\binom{n}{\lfloor \frac{n}{2}\rfloor}$ if and only if $s<s^*_m$. Modifying the Tompkins and the Ellis-Ivan-Leader constructions, we obtain upper bounds on $s_m$ and $s^*_m$.
\end{abstract}
\section{Introduction}

For a set $S$, we write $2^S$ to denote its power set and $\binom{S}{k}=\{T\subset S:|T|=k\}$.

Let $(P,\leqslant)$ be a finite poset. A \emph{weak copy} of $P$ in  $2^{[n]}$ is a family $\cG$ of sets such that there exists an order-preserving bijection $\iota:P\rightarrow \cG$, while a \emph{strong copy} requires that $p\leqslant_P q$ if and only if $\iota(p)\subseteq \iota(q)$. A family $\mathcal{F}\subseteq 2^{[n]}$ is \emph{$P$-free} (respectively, \emph{strong $P$-free}) if it contains no weak (respectively, strong) copy of $P$. The corresponding Turán functions are denoted by $\LAw(n,P)$ and $\LA(n,P)$, the maximum sizes of a weak $P$-free family and of a strong $P$-free family in $2^{[n]}$, respectively.

The study of forbidden subposet problems was initiated by Katona and Tarj\'an \cite{KatonaTarjan}, and has since become a much studied topic in extremal set theory. For surveys, see \cite{AMP,GriggsLi} or Chapter 7 of \cite{GerbnerPatkos}. A natural lower bound on both $\LAw(n,P)$ and $\LA(n,P)$ is obtained by taking the union of consecutive middle layers of the $2^{[n]}$. Let $e(P)$ (respectively, $e^*(P)$) denote the largest integer $e$ such that the union of any $e$ consecutive layers is weak $P$-free (respectively, strong $P$-free). Then
\[
\LAw(n,P)\ge \left(e(P)+o(1)\right)\binom{n}{\lfloor n/2\rfloor}
\]
and
\[
\LA(n,P)\ge \left(e^*(P)+o(1)\right)\binom{n}{\lfloor n/2\rfloor}.
\]

Motivated by numerous examples for which these lower bounds are asymptotically tight, Bukh \cite{Bukh} and Griggs and Lu \cite{GriggsLu} conjectured that
the above lower bound on $\LAw(n,P)$ is asymptotically tight and the analogous conjecture for strong copies appeared in \cite{Patkos}. In 2024, Ellis, Ivan, and Leader \cite{EllisIvanLeader} gave counterexamples to the above conjectures by the following theorem. Note that for the Boolean poset $B_d=(2^{[d]},\subseteq )$ we have $e^*(B_d)=e(B_d)=d$.

\begin{theorem}[\cite{EllisIvanLeader}]\label{EIL}
    For every $d\ge 4$ there exists a positive $\varepsilon_d$ such that if $n$ is large enough, then there exists a strong $B_d$-free family $\cF\subseteq 2^{[n]}$ of size $(d+\varepsilon_d)\binom{n}{\lfloor n/2\rfloor}$.
\end{theorem}

Very recently, Tompkins \cite{Tompkins} altered the method of Ellis, Ivan, and Leader to obtain  large strong $B_2$-free families.

\begin{theorem}[\cite{Tompkins}]\label{tomp}
    There exists a positive $\varepsilon_2$ such that if $n$ is large enough, then there exists a strong $B_2$-free family $\cF\subseteq 2^{[n]}$ of size $(2+\varepsilon_2)\binom{n}{\lfloor n/2\rfloor}$. Moreover, $\varepsilon_2$ can be chosen to be $0.14$.
\end{theorem}

In this note, we settle the case of the last remaining Boolean poset $B_3$ by using the ideas of \cite{Tompkins}. The generalized diamond poset $D_s$ has $s+2$ elements $a<b_1,b_2,\dots,b_s<c$.  Note that a weak $D_{2^d-2}$-free family is strong $B_d$-free.

\begin{theorem}\label{B3D6}
    There exists a positive $\varepsilon$ such that if $n$ is large enough, then there exists a weak $D_6$-free and therefore strong $B_3$-free family $\cF\subseteq 2^{[n]}$ of size $(3+\varepsilon)\binom{n}{\lfloor n/2\rfloor}$ and $\varepsilon$ can be chosen as $0.22$.
\end{theorem}
We then investigate the analogous extremal problems for generalized diamonds $D_s$.
For any integer $s$, let $m_s=\min\{m:2^m-2\ge s\}$,  $m^*_s=\min\{m:\binom{m}{\lfloor m/2\rfloor}\ge s\}$. Observe that $m_s$ is the smallest integer $m$ such that $B_m$ contains $s$ elements apart from $\emptyset$ and $[n]$, and $m_s^*$ is the smallest integer $m$ with $B_m$ containing an antichain of size $s$. Therefore, we have $e(D_s)=m_s,e^*(D_s)=m^*_s$. Consider the intervals $$I_m=[2^{m-1}-1,2^m-2], \hskip 1truecm I^*_m=\left[\binom{m-1}{\lfloor \frac{m-1}{2}\rfloor}+1,\binom{m}{\lfloor \frac{m}{2}\rfloor}\right].$$
Griggs, Li, and Lu \cite{GriggsLiLu} showed $\LAw(n,D_s)=(e(D_s)+o(1))\binom{n}{\lfloor n/2\rfloor}$ whenever $2^{m-1}-1\le s\le 2^m-\binom{m}{\lfloor m/2\rfloor}-1$ for some $m$. So for all but the last $\binom{m}{\lfloor m/2\rfloor}-1$ elements of $I_m$, the equality $\LAw(n,D_s)=(e(D_s)+o(1))\binom{n}{\lfloor n/2\rfloor}$ holds. Similarly, in \cite{Patkos}, it was shown that for any $\varepsilon>0$ there exists $m_0=m_0(\varepsilon)$ such that for any $m\ge m_0$ we have $\LA(n,D_s)=(m+o(1))\binom{n}{\lfloor \frac{n}{2}\rfloor}$ for all but the last $\varepsilon\binom{m}{\lfloor \frac{m}{2}\rfloor}$ elements of $I^*_m$.

The construction of Ellis, Ivan, and Leader shows that this is not true for the largest element of $I_m$ for all values of $m\ge 4$ and for the largest element of $I^*_m$ if $m\ge 4$ is even. As $D_s$ is a strong subposet of $D_{s+1}$, we have $\LAw(n,D_s)\le \LAw(n,D_{s+1}),\LA(n,D_s)\le \LA(n,D_{s+1})$. Therefore there exist  integers $s_m\in I_m, s^*_m\in I^*_m$ such that 
\begin{itemize}
    \item 
    for $s\in I_m$, we have $\LAw(n,D_s)=(m+o(1))\binom{n}{\lfloor \frac{n}{2}\rfloor}$ if and only if $s<s_m$,
    \item 
    for $s\in I^*_m$ we have $\LA(n,D_s)=(m+o(1))\binom{n}{\lfloor \frac{n}{2}\rfloor}$ if and only if $s<s^*_m$.
\end{itemize} 
Note that the result of Griggs, Li, and Lu states that $s_m>2^m-1-\binom{m}{\lfloor \frac{m}{2}\rfloor}$ and the result of \cite{Patkos} shows $s^*_m=(1-o(1))\binom{m}{\lfloor \frac{m}{2}\rfloor}$.

Modifying the constructions of Tompkins and Ellis-Ivan-Leader and sharpening the calculations of \cite{Patkos}, we obtain bounds on $s_m$ and $s^*_m$.

\begin{theorem}\label{sm}
    There exists $m_0$ such that for all $m\ge m_0$ we have $s_m\le 2^m-0.84\binom{m}{\lfloor \frac{m}{2}\rfloor}$. Furthermore, $s_3=5$, $s_4\le 11$.
\end{theorem}

Note that Theorem \ref{sm} together with the result of Griggs, Li, and Lu imply $2^m-s_m=\Theta(\binom{m}{\lfloor \frac{m}{2}\rfloor})$.

\begin{theorem}\label{smstrong}
    The following inequalities hold: $\Omega(\frac{1}{m}\binom{m}{\lfloor \frac{m}{2}\rfloor})\le \binom{m}{\lfloor \frac{m}{2}\rfloor}-s^*_m\le O(\frac{1}{m^{1/3}}\binom{m}{\lfloor \frac{m}{2}\rfloor})$.
\end{theorem}

The remainder of the paper is organized as follows. Section 2 contains the proof of Theorem \ref{B3D6}, while the proofs of Theorem \ref{sm} and Theorem \ref{smstrong} are given in Section 3.

\section{Proof of Theorem \ref{B3D6}}

We start by introducing the machinery of \cite{Tompkins}. Let us write
\[
\Delta_q=\prod_{i=1}^{\infty}(1-q^{-i}).
\]

The Gaussian binomial coefficient counting the number of $r$-dimensional subspaces of $\F_q^m$ is
\[
\genfrac{[}{]}{0pt}{}{m}{r}_q
=\prod_{i=0}^{r-1}\frac{q^m-q^i}{q^r-q^i}.
\]

\begin{proposition}[Proposition 3 in \cite{Tompkins}]\label{propRankcount}
Let $q$ be a prime power and let $0\le r\le \min\{m,s\}$.  The number of $m\times s$ matrices over $\F_q$ having rank $r$ is
\[
{m \brack r}_q\prod_{i=0}^{r-1}(q^s-q^i).
\]
\end{proposition}

\begin{cor}\label{cor}
    Let $q$ be a prime power and let $0\le r\le \min\{m,s\}$. Suppose $r=m-a$, $s=m+b$ with $a,b$ fixed and $m$ tends to infinity.  The probability that an $m\times s$ matrix over $\F_q$ with uniform independent entries has rank $r$ tends to
    \[
    \left(\frac{1}{q^{ab+\binom{a}{2}}}\prod_{i=1}^a\frac{1}{q^i-1}\prod_{j=1}^{b+a}\frac{q^j}{q^j-1}\right)\cdot \Delta_q.
    \]
\end{cor}

\begin{proof}
    By Proposition \ref{propRankcount},  the probability is 
    \[
    \frac{1}{q^{ms}}{m\brack r}_q\prod_{i=0}^{r-1}(q^s-q^i)=\frac{1}{q^{as}}{m\brack a}_q\prod_{i=0}^{r-1}(1-q^{i-s})=\prod_{h=0}^{a-1}\frac{q^{m-h}-1}{q^{m+b}}\prod_{j=1}^a\frac{1}{q^j-1}\prod_{\ell=a+b+1}^{s}(1-q^{-\ell}).
    \]
    As $r,m,s$ tend to infinity, the first product tends to $\frac{1}{q^{ab+\binom{a}{2}}}$, the third product tends to $(\prod_{j=1}^{b+a}\frac{q^j}{q^j-1})\cdot \Delta_q$.
\end{proof}

Let $k=\lfloor n/2\rfloor,
V=\F_q^k$,
and let $\phi\colon[n]\to V$ be a labelling of the ground set.  For $A\subseteq[n]$, let
\[
W(A)=W_\phi(A)=\langle \{\phi(a) : a\in A\}\rangle.
\] 

Define
\[
\mathcal \cF_{\phi,1}=\{A\in \binom{[n]}{k-1} : \dim W(A)=k-1\},
\]
\[
\mathcal \cF_{\phi,2}=\{A\in \binom{[n]}{k} : \dim W(A)\le k-1\},
\]
\[
\mathcal \cF_{\phi,3}=\binom{[n]}{k+1},
\]
and
\[
\cF_{\phi,4}=\{A\in \binom{[n]}{k+2} : \dim W(A)\ne k-1\}.
\]
Finally, set
\[
\mathcal F_\phi=\bigcup_{i=1}^4\cF_{\phi,i}.
\]

\begin{proposition}
    For any $\phi$, the family $\cF_\phi$ is weak $D_6$-free.
\end{proposition}

\begin{proof}
    Suppose otherwise. Then there exist $A\in \binom{[n]}{k-1}$, $x,y,z\notin A$ such that  all eight sets $C$ with $A\subseteq C\subseteq B=A\cup \{x,y,z\}$ belong to $\cF$. But then by definition of $\cF_{\phi,1}$ and $\cF_{\phi,2}$, we must have $W(A)=W(A\cup \{x\})=W(A\cup \{y\})=W(A\cup \{z\})$, and thus $W(A)=W(B)$ contradicting the conditions $\dim W(A)=k-1$, $\dim W(B)\neq k-1$.
\end{proof}

\begin{proof}[Proof of Theorem \ref{B3D6}]
    We use a random labeling $\phi$.
    
    Applying Corollary \ref{cor} with $m=k$, $r=k-1$, $s=k-1$ (so $a=1,b=-1$), we obtain that the probability that a $(k-1)$-set belongs to $\cF_{\phi,1}$ is $(\frac{q}{q-1}+o(1))\Delta_q$.

    Applying Corollary \ref{cor} with $m=k$, $r=k$, $s=k$ (so $a=0,b=0$), we obtain that the probability that a $k$-set does \textit{not} belong to $\cF_{\phi,2}$ is $(1+o(1))\Delta_q$.

    Applying Corollary \ref{cor} with $m=k$, $r=k-1$, $s=k+2$ (so $a=1,b=2$), we obtain that the probability that a $(k+2)$-set does \textit{not} belong to $\cF_{\phi,4}$ is $(\frac{q^4}{(q^3-1)(q^2-1)(q-1)^2}+o(1))\Delta_q$.

    By linearity of expectation, and the fact $\binom{n}{k}=(1+o(1))\binom{n}{k+j}$ for $j=-1,1,2$, we obtain that the expected size of $\cF_\phi$ has asymptotics
    \[
    \left[\frac{q}{q-1}\Delta_q +(1-\Delta_q)+1+(1-\frac{q^4}{(q^3-1)(q^2-1)(q-1)^2}\Delta_q)\right]\binom{n}{k}.
    \]
    As
    \[
    \frac{q}{q-1}-1-\frac{q^4}{(q^3-1)(q^2-1)(q-1)^2}=\frac{(q-1)(q^2-1)(q^3-1)-q^4}{(q^2-1)(q-1)^2(q^3-1)}=:f(q)
    \]
    is positive for any $q$, we obtain that there exists a $D_6$-free family of size $(3+f(q)\Delta_q)\binom{n}{\lfloor n/2\rfloor}$. Picking $q=3$ yields $\varepsilon=f(3)\Delta_3=\frac{335}{832}\Delta_3>0.22$.
\end{proof}

\section{Proofs of Theorems \ref{sm} and \ref{smstrong}}

Our approach will be similar for strong and weak $D_s$-free families. Our constructions will contain constructions of Tompkins or of Ellis-Ivan-Leader on a fixed number of middle layers surrounded by full layers so that the total number of layers used will be $m^*_s+1$ or $m_s+1$.

First, we consider the case of weak $D_s$-free families. For $m\ge 3$, let $f(m)$ denote the maximum size of a $D_2$-free family $\cF\subseteq \binom{m}{\lfloor \frac{m}{2}\rfloor-1}\cup \binom{m}{\lfloor \frac{m}{2}\rfloor}\cup \binom{m}{\lfloor \frac{m}{2}\rfloor+1}$ with the extra assumption that if $m=3$, then $\emptyset \in \cF$. Note that as $\cF$ lives only on 3 layers, all copies of $D_2$ in $\cF$ are strong copies. The following lemma together with a result of Balogh, Hu, Lidicky, and Liu \cite{BaloghHuLidickyLiu} will immediately imply Theorem \ref{sm}.

\begin{lemma}\label{extendT}
    Let $f'(m)=\sum_{i=0}^2\binom{m}{\lfloor \frac{m}{2}\rfloor-1+i}-f(m)$, the minimum number of sets in a family $\cG\subseteq \binom{m}{\lfloor \frac{m}{2}\rfloor-1}\cup \binom{m}{\lfloor \frac{m}{2}\rfloor}\cup \binom{m}{\lfloor \frac{m}{2}\rfloor+1}$ such that $\cG$ meets every copy of $D_2$ in the middle three layers. Then for any $s\in I_m$ with $2^m-2-f'(m)<s$, we have $\LAw(n,D_s)\ge (m+0.14+o(1))\binom{n}{\lfloor \frac{n}{2}\rfloor}$.
\end{lemma}

\begin{proof}
    Consider the strong $D_2$-free construction $\cF_T\subseteq \binom{n}{\lfloor \frac{n}{2}\rfloor-1}\cup \binom{n}{\lfloor \frac{n}{2}\rfloor}\cup \binom{n}{\lfloor \frac{n}{2}\rfloor+1}$ of Tompkins from \cite{Tompkins}. It has size $(2.14+o(1))\binom{n}{\lfloor \frac{n}{2}\rfloor}$. Let $\cG=\cG^-\cup \cF_T\cup \cG^+$, where $\cG^-$ consists of the $\lfloor \frac{m-2}{2}\rfloor$ layers below $\cF_T$ and $\cG^+$ consists of the $\lceil \frac{m-2}{2}\rceil$ layers above $\cF_T$. Clearly, $|\cG|=(m+0.14+o(1))\binom{n}{\lfloor \frac{n}{2}\rfloor}$. We claim that $\cG$ is weak $D_s$-free. As $\cG$ lives on $m+1=m_s+1$ layers, the bottom and top elements $G,G'$ of a copy of $D_s$ have to satisfy $|G'|-|G|=m$. Writing $[G,G']=\{H:G\subseteq H\subseteq G'\}$, as $\cF_T$ is $D_2$-free, $\cG$ can contain at most $f(m)$ sets from the middle three layers of $[G,G']$ and therefore at most $2^m-2-f'(m)$ sets from $[G,G']\setminus \{G,G'\}$.
\end{proof}

\begin{proof}[Proof of Theorem \ref{sm}]
    According to a result of Balogh, Hu, Lidicky, and Liu \cite{BaloghHuLidickyLiu} $f(m)\le (2.1511...+o(1))\binom{m}{\lfloor \frac{m}{2}\rfloor}$, so $f'(m)\ge (0.84+o(1))\binom{m}{\lfloor \frac{m}{2}\rfloor}$. This and Lemma \ref{extendT} imply the first statement of the theorem.

    Next we claim $f(3)=5$. If $\emptyset \in \cF\subseteq 2^{[3]}\setminus \{[3]\}$ is $D_2$-free, then $\cF$ cannot contain all singletons with any pair and similarly $\cF$ cannot contain all pairs and two singletons. This implies $|\cF|\le 5$ and the family $\{\emptyset, \{1\},\{2\},\{1,3\},\{2,3\}\}$ is $D_2$-free and thus shows $f(3)=5$. Then $f'(3)=2$ and so $2^3-2-f'(3)=4$, therefore by Lemma \ref{extendT} we obtain $s_3\le 5$. The result of Griggs, Li, and Lu \cite{GriggsLiLu} implies $4=2^3-1-\binom{3}{2}<s_3$.

    The statement $s_4\le 11$ follows from $f(4)=10$. The lower bound is given by the family $\binom{[4]}{1}\cup \binom{[4]}{2}$, the upper bound is a case analysis that we leave to the reader.
\end{proof}

\medskip

The complement of the family $\cF_T$ defined in \cite{Tompkins} contains sets of sizes $\lfloor \frac{n}{2}\rfloor-1,\lfloor \frac{n}{2}\rfloor,\lfloor \frac{n}{2}\rfloor+1$ and so although $\cF_T$ misses all copies of $D_2$, $\cF_T\cap [G,G']$ might contain all sets of size $\lfloor \frac{n}{2}\rfloor$ in $[G,G']$ for some $G\subset G'$ with $|G|=\lfloor \frac{n-m}{2}\rfloor, |G'|=\lceil\frac{n+m}{2}\rceil$. Therefore, the families that we used to prove Theorem \ref{sm} will not be helpful in proving Theorem \ref{smstrong}. To overcome this issue, we will use families defined in \cite{EllisIvanLeader}. We need to introduce their terminology first. An \textit{$(r,t)$-daisy} is a family $D_{r,t}$ of $\binom{t}{2}$ $r$-sets of the form $\{X\cup S: S\in \binom{T}{2}\}$ where $|X|=r-2$, $|T|=t$ and $X\cap T=\emptyset$. The Tur\'an number $\ex_r(n,D_{r,t})$ is the most number of $r$-sets in a $D_{r,t}$-free $r$-uniform hypergraph on $n$ vertices. As for all uniform Tur\'an numbers, the limit $\pi(D_{r,t})=\lim_n\frac{\ex_r(n,D_{r,t})}{\binom{n}{r}}$ exists. Bollob\'as, Leader, and Malvenuto \cite{BLM} conjectured $\pi(D_{r,t})$ tends to 0 for any fixed $t$ and $r$ tending to infinity. Ellis, Ivan, and Leader disproved this conjecture for any $t\ge 4$ and showed that $\lim_r\pi(D_{r,t})\ge 1-\frac{1}{t}-o(\frac{1}{t})$. An easy double counting argument (see \cite{EllisIvanLeader}) shows that $\frac{\ex_r(n,D_{r,t})}{\binom{n}{r}}\le 1-\frac{1}{t-1}-O(\frac{1}{n})$. Also, in \cite{BLM}, it is shown that $$\lim_{r\rightarrow \infty}\pi(D_{r,t})=\lim_{n\rightarrow \infty}\frac{\ex_{\lfloor \frac{n}{2}\rfloor}(n,D_{\lfloor \frac{n}{2}\rfloor,t})}{\binom{n}{\lfloor \frac{n}{2}\rfloor}}=\lim_{n\rightarrow \infty}\frac{\ex_{\lfloor \frac{n}{2}\rfloor+1}(n,D_{\lfloor \frac{n}{2}\rfloor+1,t})}{\binom{n}{\lfloor \frac{n}{2}\rfloor+1}}.$$ 
Based on the above, we are ready to define our $D_s$-free families. Let $t_0$ be the smallest $t$ such that $\lim_r\pi(D_{r,t})>1/2$. By results of \cite{EllisIvanLeader} there exist families
\begin{itemize}
    \item 
    $\cF_{n,4}\subseteq \binom{[n]}{\lfloor \frac{n}{2} \rfloor}$ $D_{\lfloor \frac{n}{2} \rfloor,4}$-free and $|\cF_{n,4}|=(0.29+o(1))\binom{n}{\lfloor \frac{n}{2} \rfloor}$,
    \item 
    $\cF_{n,t_0,-}\subseteq \binom{[n]}{\lfloor \frac{n}{2} \rfloor},\cF_{n,t_0,+}\subseteq \binom{[n]}{\lfloor\frac{n}{2} \rfloor+1}$ both $D_{\lfloor \frac{n}{2} \rfloor,t_0}$-free and of size $(0.5+\varepsilon+o(1))\binom{n}{\lfloor \frac{n}{2} \rfloor}$ for some fixed positive $\varepsilon$.
\end{itemize}

Then we define our final families as follows:
\begin{itemize}
    \item 
    If $m$ is even, let $\cG_{n,m}=\cF_{n,4}\cup \bigcup_{i=1}^{m/2}(\binom{[n]}{\lfloor \frac{n}{2}\rfloor -i}\cup \binom{[n]}{\lfloor \frac{n}{2}\rfloor +i})$. $|\cG_{n,m}|=(m+0.29+o(1))\binom{n}{\lfloor \frac{n}{2}\rfloor}$.
    \item 
    If $m$ is odd, let $\cG_{n,m}=\cF_{n,t_0,-}\cup \cF_{n,t_0,+}\cup \bigcup_{i=1}^{\frac{m-1}{2}}(\binom{[n]}{\lfloor \frac{n}{2}\rfloor -i}\cup \binom{[n]}{\lfloor \frac{n}{2}\rfloor +i+1})$. $|\cG_{n,m}|=(m+2\varepsilon+o(1))\binom{n}{\lfloor \frac{n}{2}\rfloor}$.
\end{itemize}
The proof of the upper bound in Theorem \ref{smstrong} will simply show that for what values $s\in I^*_m$ $\cG_{n,m}$ is strong $D_s$-free.

\begin{proof}[Proof of the upper bound in Theorem \ref{smstrong}]
    As $\cG_{n,m}$ lives on the middle $m+1$ layers, if $s>\binom{m-1}{\lfloor \frac{m-1}{2}\rfloor}$, then the top and bottom elements $G,G'$ of a strong copy of $D_s$ should come from the lowest and highest layer of $\cG_{n,m}$ as by Sperner's theorem \cite{S}, the largest antichain in $[G,G']$ has size at most $\binom{|G'|-|G|}{\lfloor \frac{|G'|-|G|}{2}\rfloor}$. Suppose $m$ is large enough. We utilize the upper bound  $\frac{\ex_r(m,D_{r,t})}{\binom{m}{r}}\le 1-\frac{1}{t-1}-O(\frac{1}{m})$
    \begin{itemize}
        \item 
        If $m$ is even, then as $\cF_{n,4}$ is $D_{\lfloor \frac{n}{2}\rfloor,4}$-free, $|\cF_{n,4}\cap [G,G']|\le (2/3+o(1))\binom{m}{ \frac{m}{2}}<\frac{3}{4}\binom{m}{ \frac{m}{2}}$. 
        \item 
        If $m$ is odd, then as  $\cF_{n,t_0,-},\cF_{n,t_0,+}$ are $D_{\lfloor \frac{n}{2}\rfloor,t_0}$-free, $|\cF_{n,t_0,-}\cap [G,G']|,|\cF_{n,t_0,+}\cap [G,G']|\le (1-\frac{1}{t_0-1}+o(1))\binom{m}{\lfloor \frac{m}{2}\rfloor}<(1-\frac{1}{t_0})\binom{m}{\lfloor \frac{m}{2}\rfloor}$. 
    \end{itemize}
    Therefore
    \begin{itemize}
        \item 
        If $m$ is even, then $\cG_{n,m}$ is strong $D_s$-free if $s>g(m)$, where $g(m)$ is the size of the largest antichain in $2^{[m]}$ that contains at most $\frac{3}{4}\binom{m}{ \frac{m}{2}}$ sets of size $\frac{m}{2}$.
        \item 
        If $m$ is odd, then $\cG_{n,m}$ is strong $D_s$-free if $s>h(m)$, where $h(m)$ is the size of the largest antichain in $2^{[m]}$ that contains at most $(1-\frac{1}{t_0})\binom{m}{ \frac{m}{2}}$ sets of size $\lfloor\frac{m}{2}\rfloor$ and $\lceil\frac{m}{2}\rceil$ each.
    \end{itemize}
So $g(m)+1$ and $h(m)+1$ are upper bounds on $s^*_m$. To prove Theorem \ref{smstrong}, we are left to give lower bounds on $\binom{m}{\frac{m}{2}}-g(m)$ and $\binom{m}{\lfloor \frac{m}{2}\rfloor}-h(m)$.

First, the LYM-inequality \cite{Lub,Mesh,Yam} states that any antichain $\cA\subseteq 2^{[m]}$ satisfies
\begin{equation}\label{lym}
    \sum_{A\in \cA}\binom{m}{|A|}^{-1}\le 1.
\end{equation}
If $m$ is even, then for any $i\neq \frac{m}{2}$, we have $\frac{\binom{m}{i}^{-1}}{\binom{m}{\frac{m}{2}}^{-1}}\ge 1+\frac{2}{m}$. So, if $\cA$ contains at most $\frac{3}{4}\binom{m}{\frac{m}{2}}$ sets of size $\frac{m}{2}$, then (\ref{lym}) implies $|\cA|\le \binom{m}{\frac{m}{2}}(\frac{3}{4}+\frac{1}{4(1+\frac{2}{m})})\le \binom{m}{\frac{m}{2}}(1-\frac{1}{4m})$ and so $\binom{m}{\frac{m}{2}}-g(m)\ge \frac{1}{4m}\binom{m}{\frac{m}{2}}$.

If $m$ is odd and $\cA\subseteq 2^{[m]}$ is antichain with $|\cA\cap \binom{[m]}{\lfloor\frac{m}{2}\rfloor}|,|\cA\cap \binom{[m]}{\lceil\frac{m}{2}\rceil}|\le (1-\frac{1}{t_0})\binom{m}{\lfloor\frac{m}{2}\rfloor}$, then we will upper bound $|\cA|$ by $(1-\Omega(\frac{1}{m}))\binom{m}{\lfloor\frac{m}{2}\rfloor}$. We will apply the Lov\'asz version \cite{Lov} of the Kruskal-Katona shadow theorem \cite{Kat,Kru}. It states that if a $k$-uniform family $\cH$ has size $\binom{x}{k}:=\prod_{i=0}^{k-1}\frac{x-i}{k-i}$ for some real $x\ge k$, then $\partial_\ell(\cH)=\{S:\exists H\in \cH, S\subseteq H, |S|=\ell\}$ has size at least $\binom{x}{\ell}$ for any $\ell \le k-1$. 

Let $\cA=\cA^-\cup\cA^+$ with $\cA^-=\{A\in \cA:|A|<\frac{m}{2}\}, \cA^+=\{A\in\cA: |A|>\frac{m}{2}\}$. As by Sperner's theorem $|\cA|\le \binom{m}{\lfloor\frac{m}{2}\rfloor}$, one of $\cA^-,\cA^+$ has size at most $\frac{1}{2}\binom{m}{\lfloor\frac{m}{2}\rfloor}$. Without loss of generality we can assume $|\cA^+|\le \frac{1}{2}\binom{m}{\lfloor\frac{m}{2}\rfloor}$. If $|\cA^+|\le \frac{1}{2t_0}\binom{m}{\lfloor\frac{m}{2}\rfloor}$, then either $|\cA|\le (1-\frac{1}{4t_0})\binom{m}{\lfloor\frac{m}{2}\rfloor}$ and we are done, or $\cA^{--}=\{A\in \cA:|A|<\frac{m}{2}-1\}$ has size at least $\frac{1}{4t_0}\binom{m}{\lfloor\frac{m}{2}\rfloor}$. In this case, an argument identical to that in the case of even $m$  using (\ref{lym}) shows that $|\cA|\le (1-\frac{1}{4t_0m})\binom{m}{\lfloor\frac{m}{2}\rfloor}$.

So we can assume $\frac{1}{2t_0}\binom{m}{\lfloor\frac{m}{2}\rfloor}\le |\cA^+|\le \frac{1}{2}\binom{m}{\lfloor\frac{m}{2}\rfloor}$. As $\cA$ is an antichain, $\cA^-$ and $\partial_{\lfloor \frac{m}{2}\rfloor}(\cA^+)$ are disjoint and $\cA^-\cup\partial_{\lfloor \frac{m}{2}\rfloor}(\cA^+)$ is an antichain and thus has size at most $\binom{m}{\lfloor \frac{m}{2}\rfloor}$. This implies $$|\cA|\le \binom{m}{\lfloor \frac{m}{2}\rfloor}-(|\partial_{\lfloor \frac{m}{2}\rfloor}(\cA^+)|-|\cA^+|).$$
Sperner's original argument implies that there exists $\cA^*$ with $|\cA^*|=|\cA^+|$ and $\cA^*\subseteq \partial_{\lceil \frac{m}{2}\rceil}(\cA^+)$. Thus $\partial_{\lfloor \frac{m}{2}\rfloor}(\cA^*)\subseteq \partial_{\lfloor \frac{m}{2}\rfloor}(\cA^+)$. Also, by the assumption $\frac{1}{2t_0}\binom{m}{\lfloor\frac{m}{2}\rfloor}\le |\cA^+|\le \frac{1}{2}\binom{m}{\lfloor\frac{m}{2}\rfloor}$, for the real $x$ with $|\cA^+|=|\cA^*|=\binom{x}{\lceil \frac{m}{2}\rceil}$ we have $m-c\le x\le m-1$, where the constant $c$ depends only on $t_0$. The ratio $\frac{\binom{x}{\lfloor \frac{m}{2}\rfloor}}{\binom{x}{\lceil \frac{m}{2}\rceil}}$ is smallest when $x=m-1$ and its minimum value is $1+\frac{2}{m}+o(\frac{1}{m})$. So by the above result of Lov\'asz, we have $|\partial_{\lfloor \frac{m}{2}\rfloor}(\cA^+)|-|\cA^+|\ge (\frac{2}{m}+o(\frac{1}{m}))|\cA^+|\ge (\frac{1}{2t_0m}+o(\frac{1}{m}))\binom{m}{\lfloor \frac{m}{2}\rfloor}$.

In all cases, we obtained $|\cA|\le (1-\Omega(\frac{1}{m}))\binom{m}{\lfloor \frac{m}{2}\rfloor}$ showing $h(m)\le  (1-\Omega(\frac{1}{m}))\binom{m}{\lfloor \frac{m}{2}\rfloor}$. This finishes the proof.
\end{proof}

Note that $\binom{[m]}{\lfloor \frac{m}{2}\rfloor-1}$ shows that $g(m),h(m)\ge \binom{m}{\lfloor \frac{m}{2}\rfloor-1}\ge (1-\frac{2}{m}+o(\frac{1}{m}))\binom{m}{\lfloor \frac{m}{2}\rfloor}$ so the order of magnitude of the upper bound on $\binom{m}{\lfloor \frac{m}{2}\rfloor}-s^*_m$ cannot be improved without further ideas.

\begin{remark}
    As ChatGPT pointed out, $t_0$ can be chosen to be 5 if instead of $\mathbb{F}_2$ as in \cite{EllisIvanLeader}, one works over $\mathbb{F}_3$: for every $x\in [n]$ pick a random label $\phi(x)\in \mathbb{F}_3^r$. Then $\cF_{\phi,r}=\{Z\in \binom{[n]}{r}:\phi(Z)$ is a basis$\}$ is $(r,5)$-daisy free, since if $X$ is the stem of the daisy, then $\mathbb{F}_3^r/\langle \phi(X)\rangle$ contains only 4 one-dimensional subspace, while there are 5 petals to the daisy, so two of them, say $x,y$, have images in the same one-dimensional subspace, so $X\cup \{x,y\}$ is not a basis. The expected density of $\cF_{\phi,\lfloor \frac{n}{2}\rfloor}$ is at least $\Delta_3=\prod_{i=1}^\infty (1-3^{-i})>0.56$.
\end{remark}

\medskip

\begin{proof}[Proof of the lower bound in Theorem \ref{smstrong}]
    The proof goes along the lines of \cite{Patkos}, just sharpens the calculations of Lemma 2.3 (ii) of \cite{Patkos}. The \textit{Lubell mass} of a family $\cF\subseteq 2^{[n]}$ is $\lambda_n(\cF)=\sum_{F\in \cF}\binom{n}{|F|}^{-1}$, the expected value of $|\cF\cap \cC|$ where $\cC$ is a maximal chain in $2^{[n]}$ taken uniformly at random from the set $\mathbf{C}_n$ of all $n!$ such chains. As $\binom{n}{|F|}^{-1}\ge \binom{n}{\lfloor \frac{n}{2}\rfloor}^{-1}$, $\lambda_n(\cF)\le c$ implies $|\cF|\le c\binom{n}{\lfloor \frac{n}{2}\rfloor}$ and a slightly more careful reasoning (see \cite[Lemma 3.2]{GriggsLiLu}) shows that if $k$ is an integer then $\lambda_n(\cF)\le k$ implies $|\cF|\le \Sigma(n,k)$, the sum of the $k$ largest binomial coefficients
    of order $n$. Thus it is enough to prove that if $\binom{m^*_s}{\lfloor \frac{m^*_s}{2}\rfloor}-s\ge \frac{C}{(m^*_s)^{1/3}}\binom{m^*_s}{\lfloor \frac{m^*_s}{2}\rfloor}$ for some large enough $C$ and $s$ itself is large enough, then for any strong $D_s$-free family $\cF\subseteq 2^{[n]}$ we have $\lambda_n(\cF)\le m^*_s$. The next lemma is essentially Lemma 2.3 (ii) of \cite{Patkos}, with a sharper calculation.
    
\begin{lemma}\label{sharper}
    There exist $C$ and $s_0$ such that the following holds. If $s\ge s_0$, $\binom{m^*_s}{\lfloor \frac{m^*_s}{2}\rfloor}-s\ge \frac{C}{(m^*_s)^{1/3}}\binom{m^*_s}{\lfloor \frac{m^*_s}{2}\rfloor}$ and $\cG\subseteq 2^{[N]}$ does not contain antichains of size $s$, then $\lambda_N(\cG)\le m^*_s$. 
\end{lemma}

\begin{proof}
    By the condition of the lemma, for every $0\le j\le N$ we have
$|\mathcal G\cap\binom{[N]}j|\le s-1$.
Consequently, 
$ \lambda_{N}(\cG)\le       \sum_{j=0}^N
        \min\{1,\frac{s-1}{\binom Nj}\}=:R_s(N).$ As $R_s(N)\le N+1$, if $N<m^*_s$, then we are done. It was proved in \cite[Lemma 2.3 (ii)]{Patkos} that for large enough $s$, $R_s(N)$ is monotone non-increasing when $N\ge m^*_s$, so it is enough to prove $R_s(m^*_s)\le m^*_s$.

        Note that for any $j$ with $|j|=O((m^*_s)^{1/3})$ we have
        \[
        \frac{\binom{m^*_s}{\lfloor \frac{m^*_s}{2}\rfloor+j}}{\binom{m^*_s}{\lfloor \frac{m^*_s}{2}\rfloor}}=\prod_{i=0}^{|j|-1}\frac{\lfloor \frac{m^*_s}{2}\rfloor-i}{\lfloor \frac{m^*_s}{2}\rfloor+i+1}=1-\frac{2j^2}{m^*_s}+O((m^*_s)^{-2/3}).
        \]
        Therefore, picking $A<B$ such that $A+B=m^*_s$, $B-A+1=\lceil (m^*_s)^{1/3}\rceil$ we have
    \[
    \lambda_{m^*_s}(\cG)\le m^*_s+1-(B-A+1)+\sum_{j=A}^B\frac{s-1}{\binom{m^*_s}{j}}\le m^*_s-(B-A)+\lceil (m^*_s)^{1/3}\rceil\frac{\binom{m^*_s}{\lfloor \frac{m^*_s}{2}\rfloor}-\frac{C}{(m^*_s)^{1/3}}\binom{m^*_s}{\lfloor \frac{m^*_s}{2}\rfloor}}{(1-\frac{1}{2(m^*_s)^{1/3}})\binom{m^*_s}{\lfloor \frac{m^*_s}{2}\rfloor}}.
    \]
    If $C>3/2$, then the last term of the RHS is at most $\lceil (m^*_s)^{1/3}\rceil-1=B-A$ and thus $\lambda_{m^*_s}(\cG)\le m^*_s$ as claimed.
\end{proof}
Let $\cF\subseteq 2^{[n]}$ be a strong $D_s$-free family. Consider the min-max partition of $\bC_n=\cup_{F\subseteq F', F,F'\in \cF}\bC_{F,F'}\cup \bC_\emptyset$, where $\bC_{F,F'}$ is the set of maximal chains $\cC\in \bC_n$ such that the smallest set of $\cF\cap \cC$ is $F$ and the largest set in $\cF\cap \cC$ is $F'$, while $\bC_\emptyset$ is the set of maximal chains that do not contain sets from $\cF$. For $F\subseteq F'$, let $\cF_{F,F'}=\{H\in \cF:F\subseteq H\subseteq F'\}$. If $F,F'\in \cF$, then as $\cF$ is strong $D_s$-free, $\cF_{F,F'}$ cannot contain antichains of size $s$, and thus Lemma \ref{sharper} implies that $\lambda_{|F'|-|F|}(\cF_{F,F'})\le m^*_s$. As $\lambda_n(\cF)$ is a weighted average of the $\lambda_{|F'|-|F|}(\cF_{F,F'})$s and 0 (the contribution of $\bC_\emptyset$), we obtain $\lambda_n(\cF)\le m^*_s$. This implies $|\cF|\le \Sigma(n,m^*_s)$ and thus $\LA(n,D_s)=\Sigma(n,m^*_s)$.
\end{proof}

Finally, let us remark that the smallest $s$ for which it is not known whether $\lim_n\frac{\LA(n,D_s)}{\binom{n}{\lfloor\frac{n}{2}\rfloor}}=m^*_s$ is 10. Also, the smallest poset $P$ for which it is not known whether $\lim_n\frac{\LA(n,P)}{\binom{n}{\lfloor\frac{n}{2}\rfloor}}=e^*(P)$ is the Harp poset $H_{1,2}$ on five elements $a,b,b',c,d$ with $a<b<c<d$, $a<b'<d$ being all its cover relations.

\end{document}